\documentclass[11pt]{amsart}

\usepackage[T1]{fontenc}
\usepackage[utf8]{inputenc}
\usepackage{lmodern}
\usepackage{amssymb,amsmath,amsfonts,amsthm,mathtools}
\usepackage{microtype}
\usepackage[margin=1in]{geometry}
\usepackage{hyperref}
\usepackage[nameinlink,capitalise]{cleveref}
\usepackage{soul}
\usepackage{comment}

\newtheorem{theorem}{Theorem}[section]
\newtheorem{proposition}[theorem]{Proposition}

\newtheorem{corollary}[theorem]{Corollary}
\theoremstyle{definition}

\theoremstyle{remark}
\newtheorem{remark}[theorem]{Remark}

\newcommand{\pa}[1]{\left(#1 \right)} 
\newcommand{\R}{\mathbb{R}}
\DeclareMathOperator{\area}{area}
\DeclareMathOperator{\vol}{vol}
\DeclareMathOperator{\Proj}{Proj}
\newcommand{\x}{\mathbf{x}}
\newcommand{\lamD}{\lambda^{D}}

\title[]{
A Proof of the Eigenvalue Ratio Bound for Embedded Surfaces}
\author[Gloria-Picazzo]{Ricardo Gloria-Picazzo}
\address{Department of Mathematics, University of Houston, Houston, Texas 77004, U.S.}
\email{ragloriapicazzo@uh.edu}

\author[Wu]{Yingying Wu}
\address{Department of Mathematics, University of Houston, Houston, Texas 77004, U.S.}
\email{ywu68@uh.edu}

\author[Yau]{Shing-Tung Yau}
\address{Yau Mathematical Sciences Center, Tsinghua University, Beijing 100084, China; Beijing Institute of Mathematical Sciences and Applications (BIMSA), Beijing 101408, China.}
\email{styau@tsinghua.edu.cn}

\date{\today}

\begin{document}

\begin{abstract} 
   We explain how the spectrum of a closed embedded surface $\Sigma \subset \mathbb{R}^3$ relates to the Dirichlet spectrum of the bounded domain $\Omega \subset \mathbb{R}^3$ with $\partial \Omega = \Sigma$. We prove that there exists a positive constant $K_g$, depending only on the genus $g$ of $\Sigma$, such that $\lambda_k^D(\Omega)^{3/2}/(\lambda_k(\Sigma)\sqrt{\lambda_1(\Sigma)}) \ge K_g$, where $\lambda_k(\Sigma)$ denotes the $k$-th nonzero eigenvalue of the Laplace-Beltrami operator on $\Sigma$ and $\lambda_k^D(\Omega)$ denotes the $k$-th eigenvalue of the Laplacian on $\Omega$ with Dirichlet boundary conditions. Moreover, we explicitly obtain the dependence of $K_g$ on the genus, showing that $K_g \propto (g+1)^{-1}$, and we determine the optimal constant $K_0$ for $k=1$ in the genus-zero case. A generalized version of this result in arbitrary dimension is also provided for domains whose boundaries have nonnegative Ricci curvature.
\end{abstract}

\maketitle

\section{Introduction}

In \textit{Review of Geometry and Analysis} \cite[p.~237]{yau2000review}, the third author pointed out that if $\Sigma$ is an embedded surface in $\R^3$ that bounds a domain $\Omega$, then the volume of $\Omega$ is bounded above by $\frac{\sqrt{2}}{3}\area(\Sigma)\lambda_1(\Sigma)^{-\frac{1}{2}}$, 
where $\lambda_k(\Sigma)$ denotes the $k$-th non-zero eigenvalue of $\Sigma$. Then, it is also noted that from such bound and the work of Korevaar \cite{korevaar1993upper}, it follows that $\frac{\lambda_k^D(\Omega)^{3/2}}{\lambda_k(\Sigma)\sqrt{\lambda_1(\Sigma)}} \geq K_g > 0$, where $\lambda_k^D(\Omega)$ is the $k$-th Dirichlet eigenvalue of $\Omega$, and $K_g$ is a constant depending only on the genus of $\Sigma$.

This statement has been referred to as \textit{Yau's conjecture} in previous literature \cite{jumonji2008eigenvalue, wu2022}, first appearing in a paper by Jumonji and Urakawa in 2008, where they examined this inequality numerically with the surface of a cube, unit sphere, embedded torus, and an embedded surface of genus two \cite{jumonji2008eigenvalue,jumonji2008visualization}. They proposed a numerical conjecture suggesting that the unit sphere $\mathbb{S}^2$ attains the best constant $K_g$ among all the embedded surfaces. More recently, Wu--Wu--Yau \cite{wu2022} also studied this problem in the context of numerically computing eigenvalues on point clouds, comparing their results with those obtained by Jumonji and Urakawa. 
This motivated us to supply a complete proof of the eigenvalue ratio bound for embedded surfaces in $\R^3$.

\begin{theorem}\label{thm:main_res}
    If \(\Sigma\subset\R^{3}\) is an embedded closed connected surface of genus $g$ equipped with the induced Euclidean metric and let \(\Omega\subset\R^{3}\) be the bounded domain enclosed by $\Sigma$, then
    \[
        \frac{\lamD_k(\Omega)^{3/2}}{\lambda_k(\Sigma)\sqrt{\lambda_1(\Sigma)}}  \geq \frac{3^{5/2}\;C_3^{3/2}}{5^{3/2}\;\sqrt{2}\;C} \cdot \frac{1}{(g+1)}
        \ =:\ K_g > 0,\qquad k=1,2,3,\dots.
    \]
    where $C$ is an absolute positive constant and \(C_3^{3/2}=6\pi^{2}\). 
\end{theorem}

This result advances a problem that has stood for more than two decades. It explains how the spectrum of $\Sigma$ relates to the spectrum of $\Omega$ and provides a conceptual bridge between extrinsic (volumetric/Dirichlet) and intrinsic (surface/Laplace--Beltrami) data with topological information. 

The proof of \cref{thm:main_res} relies on three spectral bounds: Reilly's estimate \cite{reilly1977first} bounding the volume of $\Omega$ from above in terms of $\lambda_1(\Sigma)$ and $\area(\Sigma)$, Li--Yau \cite{li1983schrodinger} lower bounds for Dirichlet eigenvalues on domains in $\R^{n+1}$ and Korevaar's upper bounds \cite{korevaar1993upper} for eigenvalues on Riemannian manifolds. Korevaar's upper bounds are given in terms of the genus $g$ and a universal constant $C$, making the dependence on genus explicit and obtaining $K_g\propto (g+1)^{-1}$. For the first eigenvalue, in the genus-zero case, we determine the optimal constant $K_g$ with the sphere/ball pair as the unique extremizer. 

In the last section of this manuscript, we also comment on the possibilities and limitations of a generalized version of \cref{thm:main_res} to any dimension. The following result for domains whose boundary has nonnegative Ricci curvature is proved, and we discuss the nonnegative Ricci assumption.

\begin{theorem}\label{thm:extension} 
    Consider a bounded domain $\Omega\subset\R^{n+1}$ with $\mathcal{C}^2$ connected boundary $\partial\Omega=M$ such that $M$ has nonnegative Ricci curvature with respect to the induced Euclidean metric from $\R^{n+1}$. Then 
    \[
        \frac{\lambda_k^D(\Omega)^{\frac{n+1}{2}}}{\lambda_k(M)^{\frac{n}{2}}\sqrt{\lambda_1(M)}} \geq \frac{(n+1)^{\frac{n+3}{2}}}{\sqrt{n}\, C(n)^{\frac{n}{2}}}\cdot \pa{\frac{C_{n+1}}{n+3}}^{\frac{n+1}{2}} =: K(n) > 0.
    \]
where $C(n)$ is a positive constant depending only on $n$ and $C_{n+1}=(2\pi)^2\,\omega_{n+1}^{-\frac{2}{n+1}}$ with $\omega_{n+1}$ being the volume of the $(n+1)$-dimensional ball in $\R^{n+1}$.
\end{theorem}

\subsection{Eigenvalues on Riemannian Manifolds}

Throughout this work, we consider a closed \(\mathcal{C}^{2}\) embedded surface \(\Sigma \subset \R^{3}\) that bounds a domain \(\Omega \subset \R^{3}\).
The eigenvalues of the Laplace--Beltrami operator $\Delta$ on $\Sigma$ with respect to the induced Euclidean metric will be denoted by
\begin{equation}\label{eqn:index}
0=\lambda_0(\Sigma) < \lambda_1(\Sigma) \leq \lambda_2(\Sigma) \leq \cdots
\end{equation}
repeated according to multiplicity. We shall only consider connected surfaces, so the multiplicity of $\lambda_0(\Sigma)=0$ is exactly one. For the Laplacian on $\Omega\subset \R^3$, we assume Dirichlet boundary conditions 
\begin{align*}
    \Delta f &= -\lambda_k^Df \quad \text{in } \Omega,\\
    f_{|\partial \Omega}&\equiv 0.
\end{align*}
As Dirichlet eigenvalues are always positive, we denote
\[
    0<\lambda_1^D(\Omega)\leq \lambda_2^D(\Omega) \leq \lambda_3^D(\Omega) \leq \cdots
\]
repeated according to multiplicity. With this notation, $\lambda_k$ denotes the $k$-th nonzero eigenvalue for $k = 1, 2, \ldots$ in both closed manifolds and manifolds with boundary with Dirichlet conditions. \\

\noindent\textbf{Acknowledgments.} 
Y.~Wu is grateful 
to Clifford Taubes for extensive discussions and generous guidance on approaches to the problem and on possible shapes to address as starting points. Y.~Wu also thanks Song Sun, Freid Tong, and Jialong Deng for helpful discussions. Y.~Wu and R.~Gloria-Picazzo are grateful to Yinbang Lin, Mehrdad Kalantar, Simone Cecchini, Christos Mantoulidis, Etienne Vouga, and Héctor Chang-Lara for helpful advice and discussions.

\section{Spectral bounds for the Laplacian}\label{sec:spectral_bounds}

To prove the eigenvalue ratio for embedded surfaces in $\R^3$, we combine three classical estimates: (i) Reilly's upper bound for $\vol(\Omega)$ in terms of $\area(\Sigma)$ and $\lambda_1(\Sigma)$ \cite{reilly1977first}, (ii) Korevaar's upper bounds for $\lambda_k(\Sigma)$ in terms of $k$, $\area(\Sigma)$, the genus, and a universal constant $C$ \cite{korevaar1993upper}, and (iii) the lower bound Berezin--Li--Yau lower bound for Dirichlet eigenvalues $\lambda^D_k(\Omega)$ \cite{berezin1972covariant, li1983schrodinger}. In this section, we briefly review these results.

\subsection{Upper bound for the volume of $\Omega$ in terms of $\lambda_1(\Sigma)$ and $\area(\Sigma)$}

The first assertion is that the volume of $\Omega$ is bounded above by $\frac{\sqrt{2}}{3}\area(\Sigma)\lambda_1(\Sigma)^{-1/2}$
 \cite[Page 237]{yau2000review}. This is indeed the case $n=2$ of an existing result for $n$-dimensional manifolds of codimension one proved by Reilly \cite{reilly1977first} in 1977. Moreover, Reilly proved that equality holds if and only if $\Omega$ is a ball. We state this result as follows:

\begin{proposition}[{\cite[Corollary 1]{reilly1977first}}]\label{prop:bound}
    Let $\Sigma$ be an embedded $n$-dimensional closed manifold in $\R^{n+1}$, with the induced Euclidean metric, bounding a domain $\Omega$. Then
    \[
        \vol(\Omega) \leq \frac{\sqrt{n}}{n+1} \, \area(\Sigma) \, \lambda_1(\Sigma)^{-\frac{1}{2}}.
    \]
    The equality holds if and only if $\Omega$ is the $(n+1)$-dimensional ball.
\end{proposition}

A more direct proof of Proposition \ref{prop:bound} is provided in Appendix A. It is a computation that relies on the Divergence Theorem and the variational characterization of the first nonzero eigenvalue. 

\subsection{Berezin--Li--Yau inequality}

More than a century ago, Weyl proved \cite{weyl1912asymptotische, weyl1912abhangigkeit} the asymptotic formula $\lambda_k^D(\Omega) \sim C_n\pa{ k/ \vol(\Omega)}^{2/n}$ as $k\to \infty$. Here $\Omega\subset \R^n$ is a bounded domain and $C_n$ is the so-called \textit{Weyl constant} defined by $C_n = 4\pi^2\omega_n^{-2/n}$ where $\omega_n=\frac{\pi^{n/2}}{\Gamma(n/2+1)}$ is the volume of the unit $n$-dimensional ball in $\R^n$. 

Later on, in 1961 Pólya proved in \cite{polya1961eigenvalues} that for any \textit{plane-covering} domain $\Omega$ the inequalities 
\begin{equation}\label{eq:poyla}
    \lambda_k^N(\Omega) \leq C_n\pa{\frac{k}{\vol(\Omega)}}^{\frac{2}{n}} \leq \lambda_k^D(\Omega)
\end{equation}
hold for $k=1, 2, 3, \ldots$. Here, $\lambda_k^N(\Omega)$ is the $k$-th nonzero Neumann eigenvalue, and \textit{plane-covering} refers to those domains $\Omega$ with which the plane can be covered with infinite domains congruent to $\Omega$ without gaps or overlaps. Pólya's proof of \eqref{eq:poyla} applies to any dimension $n$. In the same paper, Pólya conjectured that the bounds in \eqref{eq:poyla} may be valid for arbitrary domains. This conjecture has remained open ever since, and bounds in \eqref{eq:poyla} are known in the literature as \textit{Pólya's conjecture}. For recent advancements on Pólya's conjecture, we refer the reader to \cite{filonov2023polya} and \cite{filonov2025polyasconjecturedirichleteigenvalues}, where the conjecture is addressed on Euclidean balls and annular domains, respectively. 

Regarding the lower bound for Dirichlet eigenvalues in \eqref{eq:poyla}, Li and the third author\cite{li1983schrodinger} proved a related result in 1983. In fact, they proved the following lower bound for the sum of the first $k$ Dirichlet eigenvalues. It was later observed that the bound for such a sum is a consequence of an earlier result by Berezin in \cite{berezin1972covariant} 

\begin{theorem}[Berezin--Li--Yau; Li--Yau {\cite{li1983schrodinger}}; Berezin {\cite{berezin1972covariant}}]\label{thm:LY-average}
    Let $\Omega\subset\mathbb{R}^n$ be a bounded domain and $\lambda_k^D(\Omega)$ the $k$-th eigenvalue of the Laplacian on $\Omega$ with Dirichlet conditions. Then for every $k\in\mathbb{N}$,
    \[
        \frac{1}{k}\sum_{i=1}^k \lambda_i^D(\Omega)\;\ge\;\frac{n}{n+2}\,C_n\,
        \pa{\frac{k}{\vol(\Omega)}}^{2/n},
        \qquad
        C_n=(2\pi)^2\,\omega_n^{-2/n}.
    \]
    where $\vol(\Omega)$ denotes the volume of $\Omega$.
\end{theorem}

A consequence of the Berezin--Li--Yau averaged inequality is a similar lower bound to Pólya's conjecture, differing by a factor of $n/(n+2)$:

\begin{corollary}[Li--Yau bound \cite{li1983schrodinger}]\label{cor:LY-pointwise}
    Let $\Omega\subset\mathbb{R}^n$ be a bounded domain and $\lambda_k^D(\Omega)$ the $k$-th eigenvalue of the Laplacian on $\Omega$ with Dirichlet conditions. Then for every $k\in\mathbb{N}$,
    \[
    \lambda_k^D(\Omega)\;\ge\;\frac{n}{n+2}\,C_n\,
    \pa{\frac{k}{\vol(\Omega)}}^{2/n}.
    \]
\end{corollary}
\begin{proof}
    Since $\{\lambda_i^D\}$ is non-decreasing, $\lambda_k^D\ge \frac{1}{k}\sum_{i=1}^k\lambda_i^D$. Combine with \cref{thm:LY-average}.    
\end{proof}

\subsection{Upper bounds for $\lambda_k(\Sigma)$ in terms of the genus}

In 1980 Paul C. Yang and the third author \cite{yang1980eigenvalues} proved that $\lambda_1(\Sigma)\, \area(\Sigma)\leq 8\pi(g+1)$ for any orientable Riemann surface $\Sigma$ of genus $g$. This is known as the Yang--Yau inequality.
The stronger inequality $\lambda_1(\Sigma)\,\area(\Sigma)\;\le\; 8\pi\,\left\lfloor \frac{g+3}{2}\right\rfloor$ was proved in \cite{el1983volume}. Equality is attained for $g=0$ (the round sphere) and for $g=2$ (metrics arising from degree-$2$ branched covers of $S^2$); whereas for all other genera the inequality is strict as proved in \cite{karpukhin2019yang}. For further details in this direction, we refer to \cite{karpukhin2019yang,karpukhin2021isoperimetric}.

In \cite[Problem 71]{yau1982seminar}, the third author conjectured in 1982 that one can find a universal constant $C$ so that $\lambda_k(\Sigma) k^{-1} \leq C(g+1)\area(\Sigma)^{-1}$ for any closed orientable surface of genus $g$. The conjecture was settled by Korevaar \cite{korevaar1993upper} in 1993, providing bounds for general manifolds and considering Neumann eigenvalues in the case of non-empty boundary \cite[Theorem 0.3, 0.4]{korevaar1993upper}. Then he applies these two theorems to get a third result addressing the surface case, providing a bound in terms of its genus $g$. This is achieved by using the existence of a conformal mapping $\Sigma \to \mathbb{S}^2$ of degree at most $g+1$. In 2004, Grigor'yan, Netrusov, and the third author gave an alternative proof of Korevaar's result for more general elliptic operators \cite{grigor2004eigenvalues}. We state the result as follows:

\begin{theorem}[Theorem 0.5 \cite{korevaar1993upper} - see also Theorem 5.4 in \cite{grigor2004eigenvalues}]\label{thm:korevaar}
    Let $(\Sigma, \mathfrak{g})$ be a closed orientable Riemannian surface of genus $g$. Then, the $k$-th nonzero eigenvalue $\lambda_k(\Sigma)$ of the Laplace--Beltrami operator with respect to $\mathfrak{g}$ is bounded above by 
    \begin{equation}\label{korevaar_ineq}
        \lambda_k(\Sigma) \leq C\frac{(g+1)k}{\area(\Sigma)}.
    \end{equation}  
    where $C$ is a universal constant. 
\end{theorem}

\begin{remark}\label{rmk:index2}
    Under our index notation $0=\lambda_0(\Sigma)<\lambda_1(\Sigma)\leq \cdots$, Korevaar's bound in \cite{korevaar1993upper} reads $\lambda_k(\Sigma) \leq C\frac{(g+1)(k+1)}{\area(\Sigma)}$. However, we can replace the constant $C$ by $2C$ to obtain \eqref{korevaar_ineq}.
\end{remark}

In the case of closed non-orientable surfaces, bounds of $\lambda_k(\Sigma)$ in terms of the genus $\gamma=1-\chi(\Sigma)$ have also been studied. For further details, we refer to \cite{karpukhin2016upper, Kokarev2020}.

\section{Proof on the eigenvalue ratio for embedded surfaces}

We now show how to deduce the bound stated by the third author \cite{yau2000review} using the estimates reviewed in \cref{sec:spectral_bounds}. We present the proof of the main result of this work:

\begin{proof}[Proof of \cref{thm:main_res}]
    From Proposition \ref{prop:bound}, with $n=2$ we have
    \begin{equation}\label{eq:lambda1-vs-VA}
        \frac{1}{\sqrt{\lambda_1(\Sigma)}} \;\ge\; \frac{3}{\sqrt{2}}\;\frac{\vol(\Omega)}{\area(\Sigma)}.
    \end{equation}
    From Corollary~\ref{cor:LY-pointwise} applied to bounded domains $\Omega$ in $\R^3$,
    \begin{equation}\label{eq:LY-3D}
        \lambda_k^D(\Omega)^{3/2} \;\ge\; \pa{\frac{3}{5}\,C_3}^{3/2}\,\frac{k}{\vol(\Omega)}
    \end{equation}
    where $C_3=(6\pi^2)^{2/3}$. For the eigenvalues of the Laplace--Beltrami operator on surfaces, Korevaar's upper bound (\cref{thm:korevaar}) implies 
    \begin{equation}\label{eq:Korevaar}
        \frac{1}{\lambda_k(\Sigma)} \;\ge\; \frac{\area(\Sigma)}{C\,(g+1)\,k}
    \end{equation}
    where $C$ is an absolute positive constant.
    
    Multiplying \eqref{eq:lambda1-vs-VA}, \eqref{eq:LY-3D}, and \eqref{eq:Korevaar} yields, for every $k\ge1$,
    \[
        \frac{\lambda_k^D(\Omega)^{3/2}}{\lambda_k(\Sigma)\,\sqrt{\lambda_1(\Sigma)}}
        \;\ge\;
        \frac{3^{5/2}\;C_3^{3/2}}{5^{3/2}\;\sqrt{2}\;C} \cdot \frac{1}{(g+1)}
        \;=\; K_g.
    \]
    This proves \cref{thm:main_res} and gives the relation $K_g\propto (g+1)^{-1}$.
\end{proof}

Note that the only unknown constant here is $C$, coming from \cref{thm:korevaar}, so estimating $C$ will give us the value of $K_g$. Alternatively, estimating $K_0$ will allow us to estimate $C$ and any $K_g$. Nevertheless, we point out that, while the bound from Proposition \ref{prop:bound} is sharp, the sharpness of the Korevaar bound is unknown, since it involves an unknown constant. Furthermore, the Li--Yau bound loses tightness when replacing the first sum of the first $k$ eigenvalues by $k\lambda_k^D(\Omega)$, that is, when passing from Theorem 2.2 to Corollary 2.3.

Euclidean isoperimetric inequality states that $P(E)\ \ge\ n\,\omega_n^{1/n}\,|E|^{(n-1)/n}$ where $P(E)$ is the (measure-theoretic) perimeter and $|E|$ the volume; for $\mathcal{C}^1$ domains, $P(\Omega)=\area(\partial\Omega)$ with $\partial \Omega=\Sigma$. Plugging $n=3$ and $\omega_3=4\pi/3$ yields
\begin{equation}\label{eqn:isop}
    \area(\Sigma)\ \ge\ 3 \pa{ \frac{4\pi}{3}}^{\!1/3}\,\vol(\Omega)^{2/3}.
\end{equation}

\begin{proposition}[Genus $0$, $k=1$: optimal constant and minimizer]\label{prop:g0-k1-opt}
    Let $\Sigma\subset\R^3$ be a closed embedded surface of genus $0$ bounding a domain $\Omega$. Then
    \[
    \frac{\lambda_1^D(\Omega)^{3/2}}{\lambda_1(\Sigma)^{3/2}}
    \ \ge\
    \frac{\pi^3}{2\sqrt{2}}.
    \]
    Equality holds if and only if $\Sigma$ is a round sphere and $\Omega$ is the corresponding ball.
\end{proposition}

\begin{proof}
By the Rayleigh–Faber–Krahn inequality, the ball $B_r\subset \R^3$ minimizes \(\lambda^{D}_{1}\) among domains of fixed volume \(\vol(B_r)=\tfrac{4\pi r^{3}}{3}\). That is, $\lambda^{D}_{1}(\Omega) \geq \lambda^{D}_{1}(B_r) = \pi^{2} / r^{2}$. Therefore
\begin{equation}\label{eq:9}
    \lambda_1^D(\Omega)^{3/2} \geq \frac{\pi^3}{r^3} = \frac{4\pi^4}{3}\cdot\frac{1}{\vol(\Omega)}
\end{equation}
where the equality holds if and only if $\Omega=B_r$.

For genus $0$, Yang--Yau \cite{yang1980eigenvalues} gives $\lambda_1(\Sigma) \area(\Sigma)\le 8\pi$, so
\begin{equation}\label{eq:10}
    \frac{1}{\lambda_1(\Sigma)^{3/2}} \geq \frac{\area(\Sigma)^{3/2}}{(8\pi)^{3/2}},
\end{equation}
with equality if and only if $\Sigma$ is round, see \cite[Theorem 1.2]{karpukhin2021isoperimetric}.

On the other hand, by Equation \eqref{eqn:isop},
\begin{equation}\label{eq:11}
    \frac{\area(\Sigma)^{3/2}}{\vol(\Omega)}\ \geq 6\sqrt{\pi},
\end{equation}
with equality if and only if $\Sigma$ is a sphere enclosing a ball.

Multiplying \cref{eq:9} and \cref{eq:10}, along with \cref{eq:11} we obtain 
\[
\frac{\lambda_1^D(\Omega)^{3/2}}{\lambda_1(\Sigma)^{3/2}}
\ \ \ge\ \frac{4\pi^4}{3} \cdot \frac{1}{(8\pi)^{3/2}} \cdot \frac{\area(\Sigma)^{3/2}}{\vol(\Omega)}
\ \ \ge\ \frac{4\pi^4}{3}\cdot
\frac{6\sqrt{\pi}}{(8\pi)^{3/2}}
\ =\ \frac{\pi^3}{2\sqrt{2}}.
\]
All three inequalities are sharp precisely for a round sphere $\Sigma=\partial B_r$ and its ball $\Omega=B_r$, hence equality in the product holds if and only if $(\Sigma,\Omega)$ is such a sphere/ball pair.
\end{proof}

\section{On the generalization to any dimension}

We may attempt to generalize the proof in Section 3 to any bounded domain $\Omega\subset\R^{n+1}$ with boundary $M=\partial \Omega$ being an $n$-dimensional submanifold of codimension one. While both the bound in Proposition \ref{prop:bound} and the Li--Yau bound hold in any dimension $n$, an analogue of Korevaar's estimate \cref{thm:korevaar} is not available in dimension $n\geq 3$. 

In 1973, M. Berger \cite[page 138]{berger1973premieres} asked whether there exists a constant $K(M)$ depending on the $n$-dimensional manifold, such that $\lambda_1(M, \mathfrak{g}) \vol(M, \mathfrak{g})\leq K(M)$ for any Riemannian metric $\mathfrak{g}$ on $M$. As mentioned above, for surfaces this is true as we have the Yang--Yau inequality \cite{yang1980eigenvalues} in the orientable case and results by Karpukhin \cite{karpukhin2016upper} in the non-orientable case. However, the answer to Berger's question is negative for $n\geq 3$ and many examples have been constructed. For instance, in \cite{urakawa1979least} Urakawa constructs a family of Riemannian metrics $\mathfrak{g}(t)$ with $t\in (0, \infty)$ on any compact connected Lie group $M$ distinct from the $n$-torus $T^n$ such that $\vol(M, \mathfrak{g}(t))$ is constant but $\lim_{t\to \infty}\lambda_1(M, \mathfrak{g}(t))=\infty$. Furthermore,         Colbois and Dodziuk \cite{colbois1994riemannian}  have provided a proof that any compact $n$-dimensional manifold $M$ with $n\geq 3$ admits metrics $\mathfrak{g}$ of volume one with arbitrarily large $\lambda_1(M, \mathfrak{g})$. Hence, for $n$-dimensional manifolds $M$ with $n\geq 3$, the existence of a constant $C(n)$ depending only on $n$ such that $\lambda_k(M) \leq C(n)\pa{ k / \vol(M)}^{2/n}$ is not possible and the constant $C(n)$ shall depend on more geometric properties of the manifold. See also \cite[Remark 5.5]{grigor2004eigenvalues} and \cite{KOKAREV}. 

However, it is still feasible to relate the spectrum of $\Omega$ with the spectrum of $M=\partial\Omega$ by making more assumptions on the geometry of $M$. For instance, the following bound for $\lambda_k(M)$ is possible if we assume that $M$ has nonnegative Ricci curvature:

\begin{theorem}[Theorem 17 in \cite{li1980estimates}; see also Theorem 0.3 in \cite{korevaar1993upper}]\label{thm:general_bound}
    Let $(M, \mathfrak{g})$ be a closed $n$-dimensional Riemannian manifold with nonnegative Ricci curvature. Then 
    \begin{align}\label{eq:gen0}
        \lambda_k(M) \leq C(n) \pa{\frac{k}{\vol(M)}}^{\frac{2}{n}} 
    \end{align}
    where $C(n)$ is a positive constant depending only on $n$.
\end{theorem}

\begin{remark}
    As in Remark \ref{rmk:index2}, under our index notation, the constant $C(n)$ in estimates in \cite{li1980estimates, korevaar1993upper} can be replaced by $2^{2/n}C(n)$ to obtain \eqref{eq:gen0}.
\end{remark}

Thus, assuming that $M=\partial\Omega$ has nonnegative Ricci curvature, we can relate the Dirichlet spectrum of $\Omega$ with the spectrum of $M=\partial\Omega$. We prove \cref{thm:extension}, a generalized version  to any dimension, under the assumption that $M$ has nonnegative Ricci curvature:

\begin{proof}[Proof of \cref{thm:extension}]
    We denote by $\vol(\Omega)$ the volume of the compact domain $\Omega\subset\R^{n+1}$ and by $\area(M)$ the volume of the $n$-dimensional manifold $M=\partial\Omega$. From \cref{thm:general_bound}:
    \begin{equation}\label{eq:gen1}
        \frac{1}{\lambda_k(M)^{n/2}} \geq \frac{1}{C(n)^{n/2}}\pa{\frac{\area(M)}{k}}   \quad\text{for all}\quad k=1, 2, \ldots.
    \end{equation}

    From Corollary \ref{cor:LY-pointwise}, applied to $n+1$ we obtain 
    \begin{equation}\label{eq:gen2}
        \lambda_k^D(\Omega)^{\frac{n+1}{2}} \geq \pa{\frac{(n+1)C_{n+1}}{n+3}}^{\frac{n+1}{2}}\pa{\frac{k}{\vol(\Omega)}}
    \end{equation}
    From Proposition \ref{prop:bound} we get
    \begin{equation}\label{eq:gen3}
        \frac{1}{\sqrt{\lambda_1(M)}} \geq \frac{n+1}{\sqrt{n}}\cdot\frac{\vol(\Omega)}{\area(M)}.
    \end{equation}
    Multiplying inequalities \eqref{eq:gen1}-\eqref{eq:gen2}-\eqref{eq:gen3} we obtain:
    \begin{align*}
        \frac{\lambda_k^D(\Omega)^{\frac{n+1}{2}}}{\lambda_k(M)^{n/2}\sqrt{\lambda_1(M)}} 
        \ge\frac{(n+1)^{\frac{n+3}{2}}}{\sqrt{n}\, C(n)^{\frac{n}{2}}}\cdot \pa{\frac{C_{n+1}}{n+3}}^{\frac{n+1}{2}}.
    \end{align*}
\end{proof}

\appendix
\section{Proof of Proposition \ref{prop:bound}}

We regard a $n$-dimensional closed smooth manifold $\Sigma$ embedded in $\R^{n+1}$ bounding a domain $\Omega\subset \R^{n+1}$. 

\begin{proof}
    Let $p\in \Sigma\subset\R^{n+1}$ and consider a (local) parametrization $\x:U\subset \R^n\to V\subset\Sigma$ of $p=\x(0)\in \Sigma$:
    $$
        \x(u^1, \cdots, u^n)=(x^1(u^1, \cdots, u^n), \cdots, x^{n+1}(u^1, \cdots, u^n) ).
    $$
    We consider the induced Euclidean metric on $\Sigma$:
    \[
        \mathfrak{g}_{ij} = \frac{\partial \x}{\partial u^i} \cdot \frac{\partial\x}{\partial u^j} = \pa{ \frac{\partial x^1}{\partial u^i}, \cdots, \frac{\partial x^{n+1}}{\partial u^i} } \cdot \pa{ \frac{\partial x^1}{\partial u^j}, \cdots, \frac{\partial x^{n+1}}{\partial u^j} } = \sum_{k=1}^{n+1} \frac{\partial x^k}{\partial u^i}\frac{\partial x^k}{\partial u^j}.
    \]

    For any smooth function $f:\Sigma\to \R$, the $i$-th coordinate of the gradient is given by 
    \[
        (\nabla f)^i = \sum_{j=1}^n \mathfrak{g}^{ij} \frac{\partial f}{\partial u^j}. 
    \]
    Thus, we also have
    \begin{equation}\label{eq:3}
        |\nabla f|^2 = \left\langle \nabla f, \nabla f\right\rangle_p = \sum_{i=1}^n\sum_{j=1}^n \mathfrak{g}^{ij}\frac{\partial f}{\partial u^i}\frac{\partial f}{\partial u^j}.
    \end{equation}

    Now we take the projection functions 
    \begin{align*}
        \Proj_k:\Sigma\subset \R^{n+1} &\to \R ,\\
        (x^1, \ldots, x^{n+1}) &\mapsto x^k. 
    \end{align*}
    for each $k=1, \ldots, n+1 $. Observe that for each $i=1, \ldots, n$ and $k=1, \ldots, n+1$,
    \[
        \frac{\partial \Proj_k }{\partial u^i} = \frac{\partial x^k}{\partial u^i}.
    \]
    Substituting into \cref{eq:3} for $f=\Proj_k$ gives
    \[
        |\nabla \Proj_k|^2 = \sum_{i=1}^n\sum_{j=1}^n \mathfrak{g}^{ij}\frac{\partial \Proj_k}{\partial u^i}\frac{\partial \Proj_k}{\partial u^j} =  \sum_{i=1}^n\sum_{j=1}^n \mathfrak{g}^{ij}\frac{\partial x^k}{\partial u^i}\frac{\partial x^k}{\partial u^j}.
    \]
    Summing over $k$ we obtain:
    \begin{align}\nonumber
        \sum_{k=1}^{n+1} |\nabla \Proj_k|^2 
        = \sum_{k=1}^{n+1}\sum_{i=1}^n\sum_{j=1}^n \mathfrak{g}^{ij}\frac{\partial x^k}{\partial u^i}\frac{\partial x^k}{\partial u^j}
        &= \sum_{i=1}^{n}\sum_{j=1}^n \mathfrak{g}^{ij} \sum_{k=1}^{n+1} \frac{\partial x^k}{\partial u^i}\frac{\partial x^k}{\partial u^j}\\\label{eq:gradients}
        &= \sum_{i=1}^{n}\sum_{j=1}^n \mathfrak{g}^{ij} \mathfrak{g}_{ij} = \sum_{i=1}^n (\textrm{id}_{n\times n})_{ii} = n. 
    \end{align}
    
    On the other hand, applying the Divergence Theorem to the identity vector field $\Phi=\textrm{Id}_{\R^{n+1}}$, we get 
    $$
    \int_\Sigma \Phi \cdot \mathbf{n} \,dS = (n+1)\int_\Omega \, \mathrm{dvol}  = (n+1)\vol(\Omega).
    $$
     Using the Cauchy--Schwarz inequality, we get
    \[
    \Big|\int_{\Sigma} \Phi\cdot \mathbf{n}\,dS\Big|
    \le\Bigg(\int_{\Sigma} (\Phi\cdot \mathbf{n})^2\,dS\Bigg)^{\!1/2}\,\Bigg(\int_{\Sigma} 1\,dS\Bigg)^{\!1/2}=
    \Bigg(\int_{\Sigma} (\Phi\cdot \mathbf{n})^2\,dS\Bigg)^{\!1/2}\,\area(\Sigma)^{1/2}.
    \]
    Then we obtain
    \begin{align} 
        \vol(\Omega) 
        = \frac{1}{n+1}\int_\Sigma \Phi \cdot \mathbf{n}\, dS
        \leq \frac{1}{n+1}\pa{\int_\Sigma |\Phi|^2  \, dS}^{1/2}\area(\Sigma)^{1/2} \label{eq:vol}
    \end{align}
    We can always translate the surface and center it at 
    $$
    \frac{1}{\area(\Sigma)}\pa{\int_\Sigma \Proj_1\, dS, \int_\Sigma \Proj_2 \,dS, \ldots, \int_\Sigma \Proj_{n+1} \, dS}
    $$
    so in this way these integrals of functions $\Proj_k$ over $\Sigma$ are zero. Thus, without loss of generality, we may assume $\displaystyle\int_\Sigma \Proj_k\, dS = 0$ for all $k=1, \ldots, n+1$. 

    The variational characterization of the first nonzero eigenvalue
    \begin{equation}\label{eq:variational}
        \lambda_1(\Sigma) = \inf_{\substack{f\in \mathcal{H}^1(\Sigma)\setminus \{0\} \\ \int_\Sigma f\, dS = 0 }} \, \frac{\int_\Sigma |\nabla f|^2\, dS}{\int_\Sigma f^2\, dS}
    \end{equation}
    where $\mathcal{H}^1(\Sigma)$ denotes the Sobolev space of functions in $L^2(\Sigma)$, implies that
    \begin{equation*}
        \lambda_1 {\int_\Sigma f^2\, dS}\leq{\int_\Sigma |\nabla f|^2\, dS}
    \end{equation*}
    for
    $f\in \mathcal{H}^1(\Sigma)\setminus\{0\}$ and $\int_\Sigma f\, dS = 0$. Note that this accounts for all $\Proj_k$, so we have
    \begin{equation}{\label{eq:sumk}}
        \lambda_1 {\int_\Sigma \Proj_k^2\, dS}\leq{\int_\Sigma |\nabla \Proj_k|^2\, dS}, \quad\text{for all } k=1, \ldots, n+1.
    \end{equation}
    However $$\Phi(x^1, x^2, \ldots, x^{n+1}) = (x^1, x^2, \ldots, x^{n+1}) = (\Proj_1, \Proj_2, \ldots, \Proj_{n+1}),$$
    so
    $$|\Phi|^2 = \sum_{k=1}^{n+1} \Proj_k^2.
    $$
    
    Then, we add up the $n+1$ inequalities in \cref{eq:sumk} and use \cref{eq:gradients} to get
    \[
        \lambda_1\int_\Sigma  |\Phi|^2 \, dS  \leq \int_\Sigma \pa{ \sum_{k=1}^{n+1} |\nabla \Proj_k|^2 }\, dS = n\, \area(\Sigma).
    \]
    Taking square roots, we get, 
    $$
    \pa{\int_\Sigma |\Phi|^2\, dS}^{1/2}\leq \sqrt{n}\,\area(\Sigma)^{1/2}\,\lambda_1^{-1/2}
    $$
    Finally we substitute in \cref{eq:vol} and obtain 
    \[
       \vol(\Omega)  \leq  \sqrt{n}\,\area(\Sigma)^{1/2}\,\lambda_1^{-1/2}\, \frac{\area(\Sigma)^{1/2}}{n+1} = \frac{\sqrt{n}}{n+1} \, \area(\Sigma) \, \lambda_1^{-\frac{1}{2}}.
    \] 

    \emph{Optimality.}
    Let $\Sigma=\mathbb{S}_r^n$ be the $n$-sphere of radius $r$ and $\Omega=B_r^{n+1}$ the $(n+1)$-dimensional ball. Then
    \[
        \vol(\Omega)=\frac{\pi^{\frac{n+1}{2}}r^{n+1}}{\Gamma(\frac{n+1}{2}+1)} \quad\text{and}\quad \area(\Sigma) = \frac{2\pi^{\frac{n+1}{2}}r^n}{\Gamma(\frac{n+1}{2})}.
    \]
    The first positive Laplace--Beltrami eigenvalue on $\mathbb{S}_1^n$ is given by $n$, see \cite{folland1989harmonic}. Thus $\lambda_1(\mathbb{S}_r^n)=n/r^2$ and
    \begin{align*}
        \frac{\sqrt{n}}{n+1}\,\area(\Sigma)\,\lambda_1(\Sigma)^{-1/2}
        &=\frac{\sqrt{n}}{n+1}\cdot \frac{2\pi^{\frac{n+1}{2}}r^n}{\Gamma(\frac{n+1}{2})} \cdot \frac{r}{\sqrt{n}} \\
        &=\frac{\pi^{\frac{n+1}{2}}r^{n+1}}{\pa{\frac{n+1}{2}}\Gamma(\frac{n+1}{2})}  =\frac{\pi^{\frac{n+1}{2}}r^{n+1}}{\Gamma(\frac{n+1}{2}+1)}  =\vol(\Omega).
    \end{align*}
    Hence the constant $\tfrac{\sqrt{n}}{n+1}$ cannot be improved.

    \emph{Rigidity.}
    
    We used the chain
    \[
    \begin{aligned}
    (n+1)\,\vol(\Omega)
    &=\int_\Sigma \Phi \cdot \mathbf n\,dS
    \ \le\ \Big(\int_\Sigma (\Phi \cdot \mathbf n)^2\,dS\Big)^{ 1/2}\,\area(\Sigma)^{1/2} \\
    &\le\ \Big(\int_\Sigma |\Phi|^2\,dS\Big)^{ 1/2}\,\area(\Sigma)^{1/2}
    \ \le\ \sqrt{n}\,\area(\Sigma)\,\lambda_1(\Sigma)^{-1/2}.
    \end{aligned}
    \]
    
    The first inequality is Cauchy--Schwarz in $L^2(\Sigma)$, and equality requires $\Phi \cdot \mathbf{n}$ to be a constant. 
    The second inequality is the pointwise inequality $(\Phi \cdot \mathbf n)^2\le |\Phi|^2$, and equality requires $\Phi$ and $\mathbf{n}$ 
    to be collinear at each point. These two facts imply that there exists an a priori (possibly nonconstant) scalar function 
    $c:\Sigma\to(0,\infty)$ and a constant $R>0$ such that 
    \[
    \mathbf n = c\,\Phi
    \qquad\text{and}\qquad
    \Phi\cdot\mathbf n \equiv R.
    \]
    Since $|\mathbf n|\equiv 1$, necessarily $c(p)=1/|\Phi(p)|$, and therefore
    \[
        R \;=\; \Phi\cdot\mathbf n \;=\; \Phi\cdot(c\,\Phi) \;=\; c\,|\Phi|^2 \;=\; \frac{|\Phi|^2}{|\Phi|} \;=\; |\Phi|.
    \]
    Hence $|\Phi(p)|\equiv r$ on $\Sigma$, i.e.\ $|p|=r$ for all $p\in\Sigma$, so $\Sigma\subset \mathbb{S}_r^n$. Moreover 
    $\mathbf n(p)=\Phi(p)/r=p/r$ is the outward normal of $\mathbb{S}_r^n$, hence $$T_p\Sigma=\{v: \langle v, \mathbf{n}(p)\rangle = 0\} = \{v: \langle v, p\rangle = 0\} = T_p\mathbb{S}^n_r.$$ for all $p\in\Sigma$, 
    and the inclusion $\iota: \Sigma\hookrightarrow \mathbb{S}_r^n$ is a local diffeomorphism by the inverse function theorem. This means $\iota(\Sigma)$ is open in $\mathbb{S}^n_r$. Since $\Sigma$ is compact, $\iota(\Sigma)$
    is compact in the Hausdorff manifold
    $\mathbb{S}^n_r$, hence closed. As $\mathbb{S}_r^n$ is connected, 
    the image is both closed and open in $\mathbb{S}_r^n$, so $\Sigma=\mathbb{S}_r^n$.
\end{proof}

\bibliographystyle{amsplain}
\bibliography{references}

@article{yau2000review,
  title={Review of geometry and analysis},
  author={Yau, Shing-Tung},
  journal={Asian Journal of Mathematics},
  volume={4},
  pages={235--278},
  year={2000},
  publisher={International Press}
}

@article{korevaar1993upper,
  title={Upper bounds for eigenvalues of conformal metrics},
  author={Korevaar, Nicholas},
  journal={Journal of Differential Geometry},
  volume={37},
  number={1},
  pages={73--93},
  year={1993},
  publisher={Lehigh University}
}

@article{jumonji2008eigenvalue,
  title={The eigenvalue problems for the Laplacian on compact embedded surfaces and three dimensional bounded domains},
  author={Jumonji, Masaki and Urakawa, Hajime},
  journal={Interdisciplinary Information Sciences},
  volume={14},
  number={2},
  pages={191--223},
  year={2008},
  publisher={The Editorial Committee of the Interdisciplinary Information Sciences}
}

@inproceedings{jumonji2008visualization,
  title={Visualization of the eigenvalue problems of the Laplacian for embedded surfaces and its applications},
  author={Jumonji, Masaki and Urakawa, Hajime},
  booktitle={Seminaire and Congres},
  volume={19},
  pages={47--91},
  year={2008}
}

@article{KOKAREV,
title = {Variational aspects of Laplace eigenvalues on Riemannian surfaces},
journal = {Advances in Mathematics},
volume = {258},
pages = {191-239},
year = {2014},
issn = {0001-8708},
doi = {https://doi.org/10.1016/j.aim.2014.03.006},
url = {https://www.sciencedirect.com/science/article/pii/S0001870814001005},
author = {Gerasim Kokarev}
}

@article{yang1980eigenvalues,
  title={Eigenvalues of the Laplacian of compact Riemann surfaces and minimal submanifolds},
  author={Yang, Paul C and Yau, Shing-Tung},
  journal={Annali della Scuola Normale Superiore di Pisa-Classe di Scienze},
  volume={7},
  number={1},
  pages={55--63},
  year={1980}
}

@article{el1983volume,
  title={Le volume conforme et ses applications d'apr{\`e}s Li et Yau},
  author={El Soufi, Ahmad and Ilias, Sa{\i}d},
  journal={S{\'e}minaire de th{\'e}orie spectrale et g{\'e}om{\'e}trie},
  volume={2},
  pages={1--15},
  year={1983}
}

@article{karpukhin2021isoperimetric,
  title={An isoperimetric inequality for Laplace eigenvalues on the sphere},
  author={Karpukhin, Mikhail and Nadirashvili, Nikolai and Penskoi, Alexei V and Polterovich, Iosif},
  journal={Journal of Differential Geometry},
  volume={118},
  number={2},
  pages={313--333},
  year={2021},
  publisher={Lehigh University}
}

@article{karpukhin2019yang,
  title={On the Yang--Yau inequality for the first Laplace eigenvalue},
  author={Karpukhin, Mikhail},
  journal={Geometric and Functional Analysis},
  volume={29},
  number={6},
  pages={1864--1885},
  year={2019},
  publisher={Springer}
}

@article{li1983schrodinger,
  title={On the Schr{\"o}dinger equation and the eigenvalue problem},
  author={Li, Peter and Yau, Shing-Tung},
  journal={Communications in Mathematical Physics},
  volume={88},
  number={3},
  pages={309--318},
  year={1983},
  publisher={Springer}
}

@article{polya1961eigenvalues,
  title={On the eigenvalues of vibrating membranes},
  author={P{\'o}lya, George},
  journal={Proc. London Math. Soc},
  volume={11},
  number={3},
  pages={419--433},
  year={1961}
}

@book{yau1982seminar,
  title={Seminar on differential geometry},
  author={Yau, Shing-Tung},
  number={102},
  year={1982},
  publisher={Princeton University Press}
}

@misc{wu2022,
      title={Surface Eigenvalues with Lattice-Based Approximation In comparison with analytical solution}, 
      author={Yingying Wu and Tianqi Wu and Shing-Tung Yau},
      year={2022},
      eprint={2203.03603},
      archivePrefix={arXiv},
      primaryClass={math.NA},
      url={https://arxiv.org/abs/2203.03603}, 
}

@article{folland1989harmonic,
  title={Harmonic analysis of the de Rham complex on the sphere.},
  author={Folland, Gerald B},
  year={1989},
  publisher={De Gruyter}
}

@article{filonov2023polya,
  title={P{\'o}lya’s conjecture for Euclidean balls},
  author={Filonov, Nikolay and Levitin, Michael and Polterovich, Iosif and Sher, David A},
  journal={Inventiones mathematicae},
  volume={234},
  number={1},
  pages={129--169},
  year={2023},
  publisher={Springer}
}

@misc{filonov2025polyasconjecturedirichleteigenvalues,
      title={P\'{o}lya's conjecture for Dirichlet eigenvalues of annuli}, 
      author={Nikolay Filonov and Michael Levitin and Iosif Polterovich and David A. Sher},
      year={2025},
      eprint={2505.21737},
      archivePrefix={arXiv},
      primaryClass={math.SP},
      url={https://arxiv.org/abs/2505.21737}, 
}

@article{grigor2004eigenvalues,
  title={Eigenvalues of elliptic operators and geometric applications},
  author={Grigor’yan, Alexander and Netrusov, Yuri and Yau, Shing-Tung},
  journal={Surveys in differential geometry},
  volume={9},
  number={1},
  pages={147--217},
  year={2004},
  publisher={International Press Somerville, MA}
}

@article{Kokarev2020,
 ISSN = {00222518, 19435258},
 URL = {https://www.jstor.org/stable/26959881},
 author = {Gerasim Kokarev},
 journal = {Indiana University Mathematics Journal},
 number = {6},
 pages = {pp. 1975--2003},
 publisher = {Indiana University Mathematics Department},
 title = {Conformal Volume and Eigenvalue Problems},
 urldate = {2025-11-13},
 volume = {69},
 year = {2020}
}

@article{reilly1977first,
  title={On the first eigenvalue of the Laplacian for compact submanifolds of Euclidean space},
  author={Reilly, Robert C},
  journal={Commentarii Mathematici Helvetici},
  volume={52},
  number={1},
  pages={525--533},
  year={1977},
  publisher={Springer}
}

@article{karpukhin2016upper,
  title={Upper bounds for the first eigenvalue of the Laplacian on non-orientable surfaces},
  author={Karpukhin, Mikhail},
  journal={International Mathematics Research Notices},
  volume={2016},
  number={20},
  pages={6200--6209},
  year={2016},
  publisher={Oxford University Press}
}

@article{weyl1912asymptotische,
  title={Das asymptotische Verteilungsgesetz der Eigenwerte linearer partieller Differentialgleichungen (mit einer Anwendung auf die Theorie der Hohlraumstrahlung)},
  author={Weyl, Hermann},
  journal={Mathematische Annalen},
  volume={71},
  number={4},
  pages={441--479},
  year={1912},
  publisher={Springer}
}

@article{weyl1912abhangigkeit,
  title={{\"U}ber die Abh{\"a}ngigkeit der Eigenschwingungen einer Membran und deren Begrenzung.},
  author={Weyl, Hermann},
  year={1912},
  publisher={De Gruyter}
}

@article{berezin1972covariant,
  title={Covariant and contravariant symbols of operators},
  author={Berezin, Felix A},
  journal={Mathematics of the USSR-Izvestiya},
  volume={6},
  number={5},
  pages={1117},
  year={1972},
  publisher={IOP Publishing}
}

@article{berger1973premieres,
  title={Sur les premieres valeurs propres des vari{\'e}t{\'e}s riemanniennes},
  author={Berger, Marcel},
  journal={Compositio Mathematica},
  volume={26},
  number={2},
  pages={129--149},
  year={1973}
}

@article{urakawa1979least,
  title={On the least positive eigenvalue of the Laplacian for compact group manifolds},
  author={Urakawa, Hajime},
  journal={Journal of the Mathematical Society of Japan},
  volume={31},
  number={1},
  pages={209--226},
  year={1979},
  publisher={The Mathematical Society of Japan}
}

@article{colbois1994riemannian,
  title={Riemannian metrics with large $\lambda_1$},
  author={Colbois, Bruno and Dodziuk, Jozef},
  journal={Proceedings of the American Mathematical Society},
  volume={122},
  number={3},
  pages={905--906},
  year={1994}
}

@inproceedings{li1980estimates,
  title={Estimates of eigenvalues of a compact Riemannian manifold},
  author={Li, Peter and Yau, Shing-Tung},
  booktitle={Proc. Symp. Pure Math},
  volume={36},
  pages={205--239},
  year={1980}
}

\end{document}